\documentclass[11pt]{amsart}
\usepackage[margin=30mm]{geometry}
\usepackage{amsmath,amssymb}
\usepackage{amsthm}
\usepackage{mathrsfs}

\newtheorem{thm}{Theorem}

\newtheorem{lem}{Lemma}
\newtheorem{cor}{Corollary}
\newtheorem{defi}{Definition}
\newtheorem{ex}{Example}
\newtheorem{rem}{Remark}
\newtheorem*{theorem}{\it{Theorem}}

\usepackage{enumitem}

\begin{document}

\title{A note on the omega-chaos}
\author{Noriaki Kawaguchi}
\address{Department of Mathematical and Computing Science, School of Computing, Institute of Science Tokyo, 2-12-1 Ookayama, Meguro-ku, Tokyo 152-8552, Japan}
\email{gknoriaki@gmail.com}

\begin{abstract}
For any continuous self-map of a compact metric space, we provide sufficient conditions under which the infinite direct product of the map is $\omega$-chaotic. We also apply the result to obtain some examples of unusual $\omega$-chaotic maps.
\end{abstract}

\maketitle

\markboth{NORIAKI KAWAGUCHI}{A note on the omega-chaos}

{\em Chaos} is a key concept in the modern theory of dynamical systems and is relevant to unpredictable phenomena that can occur even in deterministic systems. One mathematical definition of chaos is {\em Li--Yorke chaos}, which focuses on the complicated long-term behavior of pairs of orbits \cite{LY}. A related definition is {\em $\omega$-chaos} that highlights the intricate intersections of $\omega$-limit sets \cite{SL2}. Since its introduction by S.\:Li in the context of interval maps, $\omega$-chaos and its variants have been the subject of extensive studies, revealing subtle structures in the asymptotic behavior of dynamical systems.

In this paper, we provide some criteria for $\omega$-chaos and examine its relationship with other dynamical properties. Specifically, for any continuous self-map of a compact metric space, we give sufficient conditions under which the infinite direct product of the map is $\omega$-chaotic. We also apply the result to obtain some examples of unusual $\omega$-chaotic maps, e.g., an $\omega$-chaotic map that is proximal and so not $\omega^\ast$-chaotic.

We recall the definition of {\em $\omega$-chaos} from \cite{SL2}. Throughout, $X$ denotes a compact metric space endowed with a metric $d\colon X\times X\to[0,\infty)
$. Given a continuous map $f\colon X\to X$ and $x\in X$, we denote by $\omega(x,f)$ the {\em $\omega$-limit set} of $x$ for $f$, i.e., the set of $y\in X$ such that
\[
\lim_{j\to\infty}f^{i_j}(x)=y
\]
for some $0\le i_1<i_2<\cdots$. Note that $\omega(x,f)$ is a closed subset of $X$ and satisfies $f(\omega(x,f))=\omega(x,f)$.

\begin{defi}
\normalfont
Let $f\colon X\to X$ be a continuous map. We denote by $Per(f)$ the set of {\em periodic points} for $f$:
\[
Per(f)=\bigcup_{j\ge1}\{x\in X\colon f^j(x)=x\}.
\]
A subset $S$ of $X$ is said to be an {\em $\omega$-scrambled set} for $f$ if for any $x,y\in S$ with $x\ne y$,
\begin{itemize}
\item $\omega(x,f)\setminus\omega(y,f)$ is an uncountable set,
\item $\omega(x,f)\cap\omega(y,f)\ne\emptyset$,
\item $\omega(x,f)\setminus Per(f)\ne\emptyset$.
\end{itemize}
We say that $f$ is {\em $\omega$-chaotic} if there is an uncountable $\omega$-scrambled set for $f$.
\end{defi}

The main result of this paper is the following theorem.

\begin{thm}
Given a continuous map $f\colon X\to X$, let $X^\mathbb{N}$ be the product space of infinitely many copies of $X$ and let
\[
g=f^\mathbb{N}\colon X^\mathbb{N}\to X^\mathbb{N}
\]
be the map defined by: for any $u=(u_n)_{n\ge1},v=(v_n)_{n\ge1}\in X^\mathbb{N}$, $v=g(u)$ if and only if $v_n=f(u_n)$ for all $n\ge1$. If there are a closed subset $\Lambda$ of $X$ with $f(\Lambda)\subset\Lambda$, $p\in\Lambda$, and $z\in X$ such that
\begin{itemize}
\item $\omega((p,z),f\times f)\cap\Delta\ne\emptyset$ where $\Delta=\{(x,y)\in X\times X\colon x=y\}$,
\item $\omega(z,f)\setminus\Lambda$ is an uncountable set,
\item $\omega(z,f)\setminus Per(f)\ne\emptyset$,
\end{itemize}
then $g$ is $\omega$-chaotic.
\end{thm}

Before proving Theorem 1, we state and prove two corollaries of Theorem 1.

\begin{cor}
Given a continuous map $f\colon X\to X$ and $x\in X$, let
\[
Y=\{f^i(x)\colon i\ge0\}\cup\omega(x,f),
\]
$h=f|_Y\colon Y\to Y$, and let $g=h^\mathbb{N}\colon Y^\mathbb{N}\to Y^\mathbb{N}$. If
\begin{itemize}
\item $\omega(x,f)$ is an uncountable set,
\item $\omega(x,f)\cap Per(f)\ne\emptyset$,
\item $\omega(x,f)\setminus Per(f)\ne\emptyset$,
\end{itemize}
then $g$ is $\omega$-chaotic.
\end{cor}

\begin{proof}
Let $p\in\omega(x,f)\cap Per(f)$ and take $j\ge1$ such that
\begin{itemize}
\item $f^j(p)=p$,
\item $f^i(p)\ne p$ for all $1\le i\le j-1$.
\end{itemize}
Since
\[
p\in\omega(x,f)=\bigcup_{i=0}^{j-1}\omega(f^i(x),f^j),
\]
we have $p\in\omega(f^i(x),f^j)$ for some $0\le i\le j-1$. Let $z=f^i(x)$ and
\[
\Lambda=\{f^k(p)\colon0\le k\le j-1\}.
\]
Since $f^j(p)=p$ and $p\in\omega(z,f^j)$, we have
\[
(p,p)\in\mathbb\omega((p,z),f\times f);
\]
therefore, $\omega((p,z),f\times f)\cap\Delta\ne\emptyset$. Since $\omega(z,f)=\omega(x,f)$,
\begin{itemize}
\item $\omega(z,f)\setminus\Lambda$ is an uncountable set,
\item $\omega(z,f)\setminus Per(f)\ne\emptyset$,
\end{itemize}
thus, by Theorem 1, we conclude that $g$ is $\omega$-chaotic.
\end{proof}

A continuous map $f\colon X\to X$ is said to be {\em transitive} if for any non-empty open subsets $U,V$ of $X$, there is $i\ge0$ such that $f^i(U)\cap V\ne\emptyset$. Note that $f$ is transitive if and only if $X=\omega(x,f)$ for some $x\in X$. A non-empty closed subset $M$ of $X$ with $f(M)\subset M$ is called a {\em minimal set} for $f$ if $M=\omega(x,f)$ for all $x\in M$. We say that $f$ is non-minimal if $X$ is not a minimal set for $f$.

\begin{cor}
Let $f\colon X\to X$ be a continuous map and let $g=f^\mathbb{N}\colon X^\mathbb{N}\to X^\mathbb{N}$. If $f$ is transitive and non-minimal, then $g$ is $\omega$-chaotic.
\end{cor}

\begin{proof}
Since $f$ is transitive, we have $X=\omega(z,f)$ for some $z\in X$. By the Auslander--Ellis theorem (see, e.g., \cite[Theorem 5]{Bl}), there are a minimal set $M$ for $f$ and $p\in M$ such that
\[
\liminf_{i\to\infty}d(f^i(p),f^i(z))=0
\]
and so
\[
\omega((p,z),f\times f)\cap\Delta\ne\emptyset.
\]
Note that $X\setminus M\ne\emptyset$ because $f$ is non-minimal. Since $X(=\omega(z,f))$ is perfect, i.e., has no isolated point, $X\setminus M$ is perfect and locally compact Hausdorff, thus $X\setminus M(=\omega(z,f)\setminus M)$ is uncountable. Since
\[
\omega(z,f)\setminus Per(f)=X\setminus Per(f)\ne\emptyset,
\]
letting $\Lambda=M$, we see that the assumptions of Theorem 1 are satisfied; therefore, $g$ is $\omega$-chaotic.
\end{proof}

We shall prove Theorem 1. For the proof, we recall a simplified version of the so-called Kuratowski--Mycielski theorem \cite{M}. For a topological space $Z$, a subset $S$ of $Z$ is called a {\em $G_\delta$-subset} of $Z$ if $S$ is a countable intersection of open subsets of $Z$. By the Baire category theorem, if $Z$ is a complete metric space, then every countable intersection of dense $G_\delta$-subsets of $Z$ is dense in $Z$.

\begin{theorem}[Kuratowski--Mycielski]
Let $P$ be a perfect complete separable metric space. If $R$ is a dense $G_\delta$-subset of $P\times P$, then there is a Cantor set $S$ in $P$ such that
\[
S\times S\setminus\Delta\subset R,
\]
where $\Delta=\{(x,y)\in P\times P\colon x=y\}$.
\end{theorem}

By this theorem, we obtain the following lemma.

\begin{lem}
Consider $\{p,z\}$ (with the discrete topology) and the product space
\[
P=\{p,z\}^\mathbb{N}.
\]
Then, there is a Cantor set $S$ in $P$ such that for all $s=(s_n)_{n\ge1},t=(t_n)_{n\ge1}\in S$, $s\ne t$ implies $(s_m,t_m)=(z,p)$ for some $m\ge1$. 
\end{lem}

\begin{proof}
Note that $P$ is a perfect compact metrizable space. Let $R$ be the set of
\[
(s,t)=((s_n)_{n\ge1},(t_n)_{n\ge1})\in P\times P
\]
such that $(s_m,t_m)=(z,p)$ for some $m\ge1$. We easily see that $R$ is a dense open subset of $P\times P$. By the above theorem, we obtain a Cantor set $S$ in $P$ such that for all $s,t\in S$, $s\ne t$ implies $(s,t)\in R$, i.e., $(s_m,t_m)=(z,p)$ for some $m\ge1$; therefore, the lemma has been proved.
\end{proof}

By using Lemma 1, we prove Theorem 1.

\begin{proof}[Proof of Theorem 1]
As $\{p,z\}\subset X$, we have $\{p,z\}^\mathbb{N}\subset X^\mathbb{N}$. By
\[
\omega((p,z),f\times f)\cap\Delta\ne\emptyset,
\]
we obtain $(r,r)\in\omega((p,z),f\times f)$ for some $r\in\Lambda$. It follows that
\[
(r,r,r,\dots)\in\omega(s,g)
\]
for each $s\in\{p,z\}^\mathbb{N}$; therefore $\omega(s,g)\cap\omega(t,g)\ne\emptyset$ for all $s,t\in\{p,z\}^\mathbb{N}$. Let $s=(s_n)_{n\ge1},t=(t_n)_{n\ge1}\in\{p,z\}^\mathbb{N}$ and assume that $(s_m,t_m)=(z,p)$ for some $m\ge1$. Consider the map
\[
\pi_m\colon X^\mathbb{N}\to X
\]
defined by $\pi_m(u)=u_m$ for all $u=(u_n)_{n\ge1}\in X^\mathbb{N}$. For any $w\in\omega(z,f)$, we have
\[
(q,w)\in\omega((p,z),f\times f)
\]
for some $q\in\Lambda$. We define $v=(v_n)_{n\ge1}(\in X^\mathbb{N})$ by
\begin{equation*}
v_n=
\begin{cases}
q&\text{if $s_n=p$}\\
w&\text{if $s_n=z$}
\end{cases} 
\end{equation*}
for all $n\ge1$. It follows that $v\in\omega(s,g)$ and $\pi_m(v)=w$. As $w\in\omega(z,f)$ is arbitrary, we obtain $\omega(z,f)\subset\pi_m(\omega(s,g))$. Since $t_m=p$, we have 
$\pi_m(\omega(t,g))\subset\Lambda$. By these conditions, we obtain
\[
\omega(z,f)\setminus\Lambda\subset\pi_m(\omega(s,g)\setminus\omega(t,g)).
\]
As $\omega(z,f)\setminus\Lambda$ is uncountable, so is $\omega(s,g)\setminus\omega(t,g)$. Note that $\pi_m$ satisfies
\[
\pi_m\circ g=f\circ\pi_m.
\]
If $\omega(s,g)\setminus Per(g)=\emptyset$, then by $\omega(z,f)\subset\pi_m(\omega(s,g))$, we obtain $\omega(z,f)\setminus Per(f)=\emptyset$, a contradiction. It follows that $\omega(s,g)\setminus Per(g)\ne\emptyset$. By Lemma 1, we can take a Cantor set $S$ in $\{p,z\}^\mathbb{N}$ such that for all $s=(s_n)_{n\ge1},t=(t_n)_{n\ge1}\in S$, $s\ne t$ implies $(s_m,t_m)=(z,p)$ for some $m\ge1$. Since $S$ is an $\omega$-scrambled set for $g$, we conclude that $g$ is $\omega$-chaotic, completing the proof of the theorem.
\end{proof}

By using Corollary 1, we will give several examples of $\omega$-chaotic maps with additional dynamical properties. For this purpose, we recall the notion of {\em abstract $\omega$-limit sets} and related facts \cite{B}.

Given a continuous map $f\colon X\to X$ and $\delta>0$, a finite sequence $(x_i)_{i=0}^k$ of points in $X$, where $k\ge1$, is called a {\em $\delta$-chain} of $f$ if 
\[
\sup_{0\le i\le k-1}d(f(x_i),x_{i+1})\le\delta.
\]
We say that $f$ is {\em chain transitive} if for any $x,y\in X$ and $\delta>0$, there is a $\delta$-chain $(x_i)_{i=0}^k$ of $f$ such that $x_0=x$ and $x_k=y$.

\begin{rem}
\normalfont
For a continuous map $f\colon X\to X$, the following conditions are equivalent
\begin{itemize}
\item $f$ is chain transitive,
\item for any open subset $U$ of $X$, $f(\overline{U})\subset U$ implies $U\in\{\emptyset,X\}$.
\end{itemize}
\end{rem}

\begin{rem}
\normalfont
For any continuous map $f\colon X\to X$ and $x\in X$,
\[
f|_{\omega(x,f)}\colon\omega(x,f)\to\omega(x,f)
\]
is chain transitive.
\end{rem}

\begin{defi}
\normalfont
A homeomorphism $H\colon Z\to Z$ of a compact metric space $Z$ is said to be an {\em abstract $\omega$-limit set} if there are a homeomorphism $f\colon X\to X$ and $x\in X$ such that 
\[
f|_{\omega(x,f)}\colon\omega(x,f)\to\omega(x,f)
\]
is topologically conjugate to $H\colon Z\to Z$, i.e., there is a homeomorphism $\phi\colon \omega(x,f)\to Z$ such that $\phi\circ f|_{\omega(x,f)}=H\circ\phi$.
\end{defi}

By Remark 1 and \cite[Theorem 1]{B}, we obtain the following lemma.

\begin{lem}
A homeomorphism $H\colon Z\to Z$ of a compact metric space $Z$ is an abstract $\omega$-limit set if and only if $H$ is chain transitive.
\end{lem}

Given a continuous map $f\colon X\to X$, we denote by $M(f)$ the (disjoint) union of minimal sets for $f$. A non-empty closed subset $M$ of $X$ with $f(M)\subset M$ is a minimal set for $f$ if and only if for any closed subset $S$ of $M$, $f(S)\subset S$ implies $S\in\{\emptyset,M\}$. By Zorn's lemma, there is at least one minimal set for $f$; therefore, $M(f)\ne\emptyset$.

A continuous map $f\colon X\to X$ is said to be {\em proximal} if
\[
\liminf_{i\to\infty}d(f^i(x),f^i(y))=0
\]
for all $x,y\in X$. We know that for a continuous map $f\colon X\to X$, the following conditions are equivalent
\begin{itemize}
\item $f$ is proximal,
\item $M(f)$ is a singleton, i.e., there is $p\in X$ such that $f(p)=p$ and $M(f)=\{p\}$.
\end{itemize}

\begin{rem}
\normalfont
If a continuous map $f\colon X\to X$ is proximal, then so is $f^\mathbb{N}\colon X^\mathbb{N}\to X^\mathbb{N}$.
\end{rem}

Here we recall the definition of {\em $\omega^\ast$-chaos} from \cite{SL1}.

\begin{defi}
\normalfont
Let $f\colon X\to X$ be a continuous map. A subset $S$ of $X$ is said to be an {\em $\omega^\ast$-scrambled set} for $f$ if for any $x,y\in S$ with $x\ne y$,
\begin{itemize}
\item $\omega(x,f)\setminus\omega(y,f)$ contains an infinite minimal set for $f$,
\item $\omega(x,f)\cap\omega(y,f)\ne\emptyset$.
\end{itemize}
We say that $f$ is {\em $\omega^\ast$-chaotic} if there is an uncountable $\omega^\ast$-scrambled set for $f$.
\end{defi}

\begin{rem}
\normalfont
Given a continuous map $f\colon X\to X$, it is obvious that an $\omega^\ast$-scrambled set for $f$ is an $\omega$-scrambled set for $f$, thus if $f$ is $\omega^\ast$-chaotic, then $f$ is $\omega$-chaotic. If $f$ is proximal, then every $\omega^\ast$-scrambled set $S$ for $f$ satisfies $|S|\le1$; therefore, $f$ is not $\omega^\ast$-chaotic.
\end{rem}

\begin{rem}
\normalfont
In Section 6 of \cite{LKKF}, an example is given of a continuous map $f\colon X\to X$ such that
\begin{itemize}
\item $X=\overline{Per(f)}$,
\item $f$ is mixing,
\item $f$ is $\omega$-chaotic but not $\omega^\ast$-chaotic.
\end{itemize}
\end{rem}

We can now present the promised examples.

\begin{ex}
\normalfont
Let $Z=S^1=\{(a,b)\in\mathbb{R}^2\colon a^2+b^2=1\}$ and let $q=(0,1)$. We consider the standard bijection $i\colon\mathbb{R}\to Z\setminus\{q\}$ defined as
\[
i(r)=\left(\frac{2r}{r^2+1},\frac{r^2-1}{r^2+1}\right)
\]
for all $r\in\mathbb{R}$, and define a homeomorphism $H\colon Z\to Z$ by
\begin{itemize}
\item $H(q)=q$,
\item $H(i(r))=i(r+1)$ for all $r\in\mathbb{R}$.
\end{itemize}
Since $H$ is chain transitive, by Lemma 2,  there are a homeomorphism $f\colon X\to X$ and $x\in X$ such that 
\[
f|_{\omega(x,f)}\colon\omega(x,f)\to\omega(x,f)
\]
is topologically conjugate to $H\colon Z\to Z$, i.e., there is a homeomorphism $\phi\colon\omega(x,f)\to Z$ such that $\phi\circ f|_{\omega(x,f)}=H\circ\phi$. Consider $Y$, $h$, $g$ as in Corollary 1. Letting $p=\phi^{-1}(q)$, we see that
\begin{itemize}
\item $f(p)=p$,
\item $\omega(x,f)$ is an uncountable set,
\item $\omega(x,f)\cap Per(f)=\{p\}$; therefore, $\omega(x,f)\cap Per(f)\ne\emptyset$ and $\omega(x,f)\setminus Per(f)\ne\emptyset$.
\end{itemize}
It follows that $g$ is $\omega$-chaotic. Since $M(h)=\{p\}$, a singleton, $h$ is proximal and so is $g$. It follows that $g$ is not $\omega^\ast$-chaotic. It is obvious that the topological entropy of $g$ is zero.
\end{ex}

A continuous map $f\colon X\to X$ is said to be 
\begin{itemize}
\item {\em weakly mixing} if $f\times f\colon X\times X\to X\times X$ is transitive,
\item {\em mixing} if for any non-empty open subsets $U,V$ of $X$, there is $j\ge0$ such that $f^i(U)\cap V\ne\emptyset$ for all $i\ge j$.
\end{itemize}
We say that $f$ is {\em uniformly rigid} if for any $\epsilon>0$, there is $j\ge1$ such that
\[
\sup_{x\in X}d(x,f^j(x))\le\epsilon.
\]
Let $g=f^\mathbb{N}\colon X^\mathbb{N}\to X^\mathbb{N}$ and note that
\begin{itemize}
\item as a consequence of Furstenberg's $2$ implies $n$ theorem (see \cite[Proposition I\!I.3]{F}), if $f$ is weakly mixing, then so is $g$,
\item if $f$ is mixing, then so is $g$,
\item if $f$ is uniformly rigid, then so is $g$.
\end{itemize}

\begin{ex}
\normalfont
Given a transitive continuous map $f\colon X\to X$ such that $X$ is an infinite set, we fix $x\in X$ with $X=\omega(x,f)$. Let $Y$, $h$, $g$ be as in Corollary 1. Note that, in this case, $Y=X$ and $h=f$. 
\begin{itemize}
\item[(1)] If $Per(f)\ne\emptyset$, in particular, if $f$ is Devaney chaotic, then $g$ is $\omega$-chaotic.
\item[(2)] As an example given in Section 5 of \cite{FHLO}, if $f$ is proximal, weakly mixing, and uniformly rigid, then $g$ is proximal, weakly mixing, uniformly rigid, $\omega$-chaotic, and not $\omega^\ast$-chaotic.
\item[(3)] As an example constructed in Section 4 of \cite{O}, if $f$ is proximal and mixing, then $g$ is proximal, mixing, $\omega$-chaotic, and not $\omega^\ast$-chaotic.
\end{itemize}
\end{ex}

\begin{ex}
\normalfont
Let $Z=S^1=\{(a,b)\in\mathbb{R}^2\colon a^2+b^2=1\}$. We define a homeomorphism $H\colon Z\to Z$ by $H(c)=c$ for all $c\in Z$. Since $H$ is chain transitive, by Lemma 2,  there are a homeomorphism $f\colon X\to X$ and $x\in X$ such that 
\[
f|_{\omega(x,f)}\colon\omega(x,f)\to\omega(x,f)
\]
is topologically conjugate to $H\colon Z\to Z$. Note that $f(p)=p$ for all $p\in\omega(x,f)$. Consider $Y$, $h$, $g$ as in Corollary 1. As in the proof of Theorem 1, letting $p\in\omega(x,f)$ and $\Lambda=\{p\}$, we have a Cantor set $S$ in $Y^\mathbb{N}$ such that for any $s,t\in S$ with $s\ne t$,
\begin{itemize}
\item $\omega(s,g)\setminus\omega(t,g)$ is an uncountable set,
\item $(p,p,p,\dots)\in\omega(s,g)\cap\omega(t,g)$ and so $\omega(s,g)\cap\omega(t,g)\ne\emptyset$.
\end{itemize}
However, since $g(v)=v$ for all $u\in Y^\mathbb{N}$ and $v\in\omega(u,g)$, $g$ is not $\omega$-chaotic.
\end{ex}

\end{document}